\documentclass{article}
\usepackage{amsmath,amssymb,amsthm}
\usepackage{cases}
\usepackage{tikz}
\usepackage{mathtools}
\usepackage{stmaryrd}
\usepackage{youngtab}
\usepackage{enumerate}
\usepackage{caption}
\usepackage{graphicx}
\usetikzlibrary{intersections, calc, arrows.meta}
\newcommand{\mcS}{\mathcal{S}}

\newcommand{\mcG}{\mathcal{G}}
\newcommand{\mcF}{\mathcal{F}}
\newcommand{\mcOH}{\mathcal{O}}
\newcommand{\mcY}{\mathcal{Y}}
\newcommand{\mcH}{\mathcal{H}}
\newcommand{\mcP}{\mathcal{P}}
\newcommand{\mcC}{\mathcal{C}}
\newcommand{\mcN}{\mathcal{N}}
\newcommand{\la}{\langle}
\newcommand{\ra}{\rangle}
\newcommand{\lala}{\lambda = (\lambda_{1} , \ldots , \lambda_{m})}
\newcommand{\BN}{\mathbb{N}}
\newcommand{\BZ}{\mathbb{Z}}

\newcommand{\BR}{\mathbb{R}}

\newcommand{\BD}{\mathbb{D}}
\newcommand{\xra}{\xrightarrow}
\newcommand{\ol}{\overline}

\newcommand{\bs}{\setminus}
\newcommand{\bracket}[1]{\llbracket #1 \rrbracket}

\allowdisplaybreaks
\theoremstyle{definition}
\newtheorem{thm}{Theorem}[section]
\newtheorem{defi}[thm]{Definition}
\newtheorem{lemm}[thm]{Lemma}
\newtheorem{prop}[thm]{Proposition}
\newtheorem{exam}[thm]{Example}
\newtheorem{rema}[thm]{Remark}
\newtheorem{coro}[thm]{Corollary}

\numberwithin{equation}{section}
\title{Game positions of

Multiple Hook Removing Game}
\author{Yuki Motegi}
\date{}
\begin{document}
\maketitle
\begin{abstract}
Multiple Hook Removing Game (MHRG for short) introduced in \cite{AT} is an impartial game played in terms of Young diagrams. In this paper, we give a characterization of the set of all game positions in MHRG. As an application, we prove that for $t \in \BZ_{\geq 0}$ and $m, n \in \BN$ such that $t \leq m \leq n$, and a Young diagram $Y$ contained in the rectangular Young diagram $Y_{t,n}$ of size $t \times n$, $Y$ is a game position in MHRG with $Y_{m,n}$ the starting position if and only if $Y$ is a game position in MHRG with $Y_{t,n-m+t}$ the starting position, and also that the Grundy value of $Y$ in the former MHRG is equal to that in the latter MHRG.
\end{abstract}


\section{Introduction.}\label{sec1}

The Sato-Welter game is an impartial game studied by Welter \cite{Welter} and Sato \cite{Sato1}, independently. This game is played in terms of Young diagrams. The rule is given as follows:
\begin{enumerate}[(i)]
\item The starting position is a Young diagram $Y$.
\item Assume that a Young diagram $Y'$ appears as a game position. A player chooses a box $(i,j) \in Y'$, and moves game position from $Y'$ to $Y'\la i,j \ra$, where $Y'\la i,j \ra$ is the Young diagram which is obtained by removing the hook at $(i,j)$ from $Y'$ and filling the gap between two diagrams (see Figure~\ref{fig} below). \label{ii}
\item The (unique) ending position is the empty Young diagram $\emptyset$. The winner is the player who makes $\emptyset$ after his/her operation (\ref{ii}).
\end{enumerate}
Kawanaka \cite{Ka} introduced the notion of a plain game, as a generalization of the Sato-Welter game. A plain game is played in terms of $d$-complete posets which was introduced and classified by Proctor \cite{P1,P2}, and can be thought of as a generalization of Young diagrams. It is known that $d$-complete posets are closely related to not only the combinatorial game theory, but also the representation theory and the algebraic geometry associated with simply-laced finite-dimensional simple Lie algebras. For example, the weight system of a minuscule representation (which is identical to the Weyl group orbit of a minuscule fundamental weight) for a simply-laced finite-dimensional simple Lie algebra can be described in terms of a $d$-complete poset. Applying the ``folding'' technique to this fact for the simply-laced case, Tada \cite{Tada} described the Weyl group orbits of some fundamental weights for multiply-laced finite-dimensional simple Lie algebras in terms of $d$-complete posets with ``coloring''.
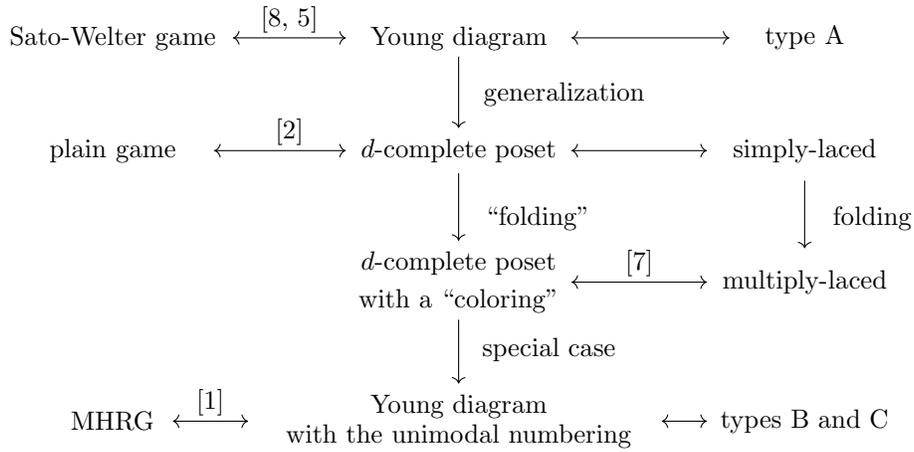
\begin{center}
\begin{tikzpicture}
\draw(0,0)node{$d$-complete poset};
\draw(0,-0.5)node{with a ``coloring''};
\draw(4.6,-0.25)node{multiply-laced};
\draw[<->] (1.5,-0.25)--(3.3,-0.25);
\draw(2.4,0)node{\cite{Tada}};
\draw(-4.6,1.5)node{plain game};
\draw[<->](-3.25,1.5)--(-1.5,1.5);
\draw(-2.25,1.75)node{\cite{Ka}};
\draw(0,1.5)node{$d$-complete poset};
\draw[<->] (1.5,1.5)--(3.3,1.5);
\draw(4.6,1.5)node{simply-laced};
\draw[->](0,1.2)--(0,0.3);
\draw[->](4.6,1.1)--(4.6,0.2);
\draw(5.5,0.6)node{folding};
\draw(1.05,0.6)node{``folding''};
\draw(0,3)node{Young diagram};
\draw[->](0,2.75)--(0,1.8);
\draw(1.4,2.25)node{generalization};
\draw(-4.6,3)node{Sato-Welter game};
\draw[<->](-3,3)--(-1.5,3);
\draw(-2.25,3.25)node{\cite{Welter,Sato1}};
\draw[<->](1.5,3)--(3.6,3);
\draw(4.6,3)node{type A};
\draw[->](0,-0.75)--(0,-1.6);
\draw(1.2,-1.15)node{special case};
\draw(0,-1.9)node{Young diagram};
\draw(0,-2.3)node{with the unimodal numbering};
\draw[<->](-2.8,-2.1)--(-3.8,-2.1);
\draw(-4.6,-2.1)node{MHRG};
\draw(-3.3,-1.85)node{\cite{AT}};
\draw(4.6,-2.1)node{types B and C};
\draw[<->](3.3,-2.1)--(2.7,-2.1);
\end{tikzpicture}
\end{center}

Based on \cite{Tada}, Abuku and Tada \cite{AT} introduced a new impartial game, named Multiple Hook Removing Game (MHRG for short). MHRG is played in terms of Young diagrams with the unimodal numbering; for the definition of unimodal numbering, see Section~\ref{sec3} and Example~\ref{EMN}. Let us explain the rule of MHRG. We fix positive integers $m, n \in \BN$ such that $m \leq n$. Let $Y_{m,n} \coloneqq \{(i,j) \in \BN^{2} \mid 1 \leq i \leq m,\, 1 \leq j \leq n\}$ be the rectangular Young diagram of size $m \times n$. We denote by $\mcF (Y_{m,n})$ the set of all Young diagrams contained in the rectangular Young diagram $Y_{m,n}$. For a game position $G$ of an impartial game, we denote by $\mcOH (G)$ the set of all options of $G$. The rule of MHRG is given as follows:
\begin{enumerate}[(1)]
\item All game positions are some Young diagrams contained in $\mcF(Y_{m,n})$ with the unimodal numbering. The starting position is the rectangular Young diagram $Y_{m,n}$.
\item Assume that $Y \in \mcF(Y_{m,n})$ appears as a game position. If $Y \neq \emptyset$ (the empty Young diagram), then a player chooses a box $(i,j) \in Y$, and remove the hook at $(i,j)$ in $Y$. We denote by $Y\la i,j \ra$ the resulting Young diagram. Then we know from \cite[Lemma~3.15]{AT} (see also Lemma~\ref{at most} below) that $f \coloneqq \# \{(i',j') \in Y\la i,j \ra \mid \mcH_{Y\la i,j \ra}(i',j') = \mcH_{Y}(i,j)\ \text{(as multisets)}\} \leq 1$ ,where $\mcH_{Y}(i,j)$ (resp., $\mcH_{Y\la i,j \ra}(i',j')$) is the numbering multiset for the hook at $(i,j) \in Y$ (resp., $(i',j') \in Y\la i,j \ra$); see Section~\ref{sec3}. If $f = 0$, then a player moves $Y$ to $Y\la i,j \ra \in \mcOH(Y)$. If $f = 1$, then a player moves $Y$ to $(Y\la i,j \ra)\la i',j' \ra \in \mcOH(Y)$, where $(i', j') \in Y\la i,j \ra$ is the unique element such that $\mcH_{Y\la i,j \ra}(i', j') = \mcH_{Y}(i,j)$.
\item The (unique) ending position is the empty Young diagram $\emptyset$. The winner is the player who makes $\emptyset$ after his/her operation (2).
\end{enumerate}

In general, not all Young diagrams in $\mcF(Y_{m,n})$ appear as game positions of MHRG (see Example~\ref{Emotiv}). The goal of this paper is to give a characterization of the set of all game positions in MHRG. Let us explain our results more precisely. Let $\binom{[1,m+n]}{m}$ denote the set of all subsets of $[1,m+n] \coloneqq \{x \in \BN \mid 1 \leq x \leq m+n\}$ having $m$ elements. Then there exists a bijection $I$ from $\mcF (Y_{m,n})$ onto $\binom{[1,m+n]}{m}$ (see Subsection~\ref{subsec21} below). Let $Y^{D}$ denote the dual Young diagram of $Y$ in $Y_{m,n}$ (see Subsection~\ref{subsec21}).  We set $c \coloneqq (m+n-1+\chi)\, /\, 2$, where $\chi = 0$ (resp., $\chi = 1$) if $m+n$ is odd (resp., even). For $Y \in \mcF (Y_{m,n})$, we set $I_{R}(Y) \coloneqq I(Y)\, \cap\, [c+1-\chi, m+n]$. We denote by $\mcS (Y_{m,n})$ the set of all those Young diagrams in $\mcF (Y_{m,n})$ which appear as game positions of MHRG (with $Y_{m,n}$ the starting position).
\begin{thm}[= Theorem~\ref{main}]\label{intro0}
Let $Y \in \mcF(Y_{m,n})$, and $\lala$ the partition corresponding to $Y$. The following (I), (I\hspace{-0.2mm}I), (I\hspace{-0.2mm}I\hspace{-0.2mm}I), and (I\hspace{-0.2mm}V) are equivalent. \\
(I) $Y \in \mcS(Y_{m,n})$. \hspace{5mm} (I\hspace{-0.2mm}I) $Y^{D} \in \mcS(Y_{m,n})$. \hspace{5mm} (I\hspace{-0.2mm}I\hspace{-0.2mm}I) $I_{R}(Y) \cap I_{R}(Y^{D}) = \emptyset$. \\
(I\hspace{-0.2mm}V) $\lambda_{i}+\lambda_{j} \neq n-m+i+j-1$ for all $1 \leq i \leq j \leq m$.
\end{thm}

\begin{thm}[= Theorem~\ref{t rows}]\label{intro1}
Let $t \in \BZ_{\geq 0}$ and $m,n \in \BN$ such that $t \leq m \leq n$. For a Young diagram $Y$ having at most $t$ rows, $Y \in \mcS(Y_{m,n})$ if and only if $Y \in \mcS(Y_{t,n-m+t})$. Moreover, the Grundy value of $Y$ as an element of $\mcS(Y_{m,n})$ is equal to the Grundy value of $Y$ as an element of $\mcS(Y_{t,n-m+t})$.
\end{thm}

This paper is organized as follows. In Section~\ref{sec2}, we fix our notation for Young diagrams, and recall some basic facts on the combinatorial game theory. In Section~\ref{sec3}, we recall the definition of the unimodal numbering and the diagonal expression for Young diagrams. In Section~\ref{sec4}, we recall the rule of MHRG, and a basic property (Lemma~\ref{at most}). In Sections~\ref{sec5} and \ref{sec6}, we prove Theorems~\ref{intro0} and~\ref{intro1} above, respectively.

\subsubsection*{Acknowledgements :}
The author would like to thank Daisuke Sagaki, who is his supervisor, for useful discussions. He also thanks Tomoaki Abuku and Masato Tada for valuable comments.

\section{Preliminaries.}\label{sec2}
\subsection{Young diagrams.}\label{subsec21}
Let $\BN$ denote the set of positive intgers. For $a,b \in \BZ$, we set $[a,b] \coloneqq \{x \in \BZ \mid a \leq x \leq b\}$. Throughout this paper, we fix $m ,n \in \BN$ such that $m \leq n$. For a positive integer $x \in \BN$, we set $\ol{x} \coloneqq m+n+1-x$. Let $\mcY_{m}(m+n)$ be the set of partitions $\lala$ of length at most $m$ such that $n \geq \lambda_{1} \geq \cdots \geq \lambda_{m} \geq 0$. We can identify $\lala \in \mcY_{m}(m+n)$ with the {\it Young diagram} $Y_{\lambda} \coloneqq \{(i, j) \in \BN^{2} = \BN \times \BN \mid 1 \leq i \leq m, 1 \leq j \leq \lambda_{i}\}$ of shape $\lambda$; if $\lambda = (0 , 0 , \ldots , 0) \in \mcY_{m}(m+n)$, then we denote $Y_{\lambda}$ by $\emptyset$, and call it the {\it empty Young diagram}.
We identify $(i,j) \in Y_{\lambda}$ with the square in $\BR^{2}$ whose vertices are $(i-1, j-1), (i-1, j), (i, j-1),$ and $(i, j)$; elements in $Y_{\lambda}$ are called {\it boxes} in $Y_{\lambda}$. Let $Y_{m,n} \coloneqq \{(i, j) \in \BN^{2} \mid 1 \leq i \leq m, 1 \leq j \leq n\}$ be the rectangular Young diagram of size $m \times n$, which corresponds to $(n , n , \ldots , n) \in \mcY_{m}(m+n)$. Set $\mcF(Y_{m,n}) \coloneqq \{Y_{\lambda} \mid \lambda \in \mcY_{m}(m+n)\}$; notice that $\mcF(Y_{m,n})$ is identical to the set of all Young diagrams contained in the rectangular Young diagram $Y_{m,n}$. We set $\lambda^{D} \coloneqq (n-\lambda_{m} , \ldots , n-\lambda_{1}) \in \mcY_{m}(m+n)$. The Young diagram $Y_{\lambda}^{D} \coloneqq Y_{\lambda^{D}}$ is called the {\it dual Young diagram} of $Y_{\lambda}$ (in $Y_{m,n}$).
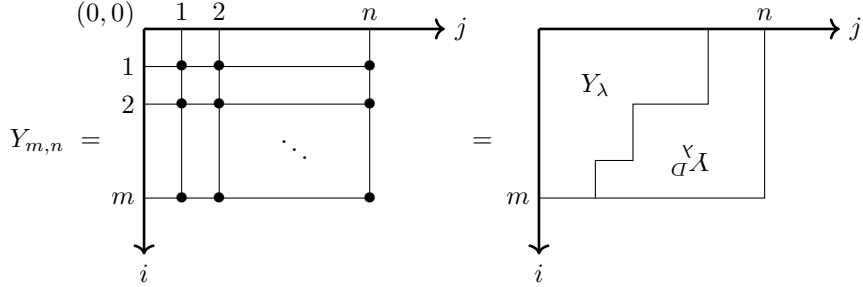
\begin{center}
\begin{tikzpicture}
\draw(-1,-1.5)node{$Y_{m,n}\ =\ \ \ $};
\draw(0,0.2)node[above,left]{$(0,0)$};
\draw(0.5,0)node[above]{$1$};
\draw(0,-0.5)node[left]{$1$};
\draw(1,0)node[above]{$2$};
\draw(0,-1)node[left]{$2$};
\draw[line width=1pt,->] (0,0)--(4,0)node[right]{$j$};
\draw[line width=1pt,->] (0,0)--(0,-3)node[below]{$i$};
\draw(0.5,-0.5)node{$\bullet$};
\draw(0.5,-1)node{$\bullet$};
\draw(1,-0.5)node{$\bullet$};
\draw(1,-1)node{$\bullet$};
%\draw(2,-0.75)node{$\cdots \cdots$};
%\draw(0.75,-1.5)node{$\vdots$};
\draw(2,-1.5)node{$\ddots$};
\draw(0.5,-2.25)node{$\bullet$};
\draw(1,-2.25)node{$\bullet$};
\draw(0,-2.25)node[left]{$m$};
%\draw(2,-2.25)node{$\cdots \cdots$};
\draw(3,0)node[above]{$n$};
\draw(3,-0.5)node{$\bullet$};
\draw(3,-1)node{$\bullet$};
%\draw(3,-1.5)node{$\vdots$};
\draw(3,-2.25)node{$\bullet$};

\draw(0,-0.5)--(3,-0.5);
\draw(0,-1)--(3,-1);
\draw(0,-2.25)--(3,-2.25);
\draw(0.5,0)--(0.5,-2.25);
\draw(1,0)--(1,-2.25);
\draw(3,0)--(3,-2.25);

\draw(4.5,-1.5)node{$=$};

\draw[line width=1pt,->] (5.25,0)--(9.25,0)node[right]{$j$};
\draw[line width=1pt,->] (5.25,0)--(5.25,-3)node[below]{$i$};
\draw(5.25,-2.25)node[left]{$m$}--(8.25,-2.25)--(8.25,0)node[above]{$n$};
\draw(6,-2.25)--(6,-1.75)--(6.5,-1.75)--(6.5,-1)--(7.5,-1)--(7.5,0);
\draw(6,-0.75)node{$Y_{\lambda}$};
\draw(7.25,-1.75)node{\rotatebox{180}{$Y_{\lambda}^{D}$}};
\end{tikzpicture}
\end{center}

Let $\binom{[1,m+n]}{m}$ denote the set of all subsets of $[1,m+n]$ having $m$ elements. For $\lala \in \mcY_{m}(m+n)$, we set $i_{t}' \coloneqq \lambda_{m-t+1} + t$ for $1 \leq t \leq m$; observe that $I_{\lambda} \coloneqq \{i_{1}' < \cdots < i_{m}'\} \in \binom{[1,m+n]}{m}$. It is well-known that the map $\lambda \mapsto I_{\lambda}$ is a bijection from $\mcY_{m}(m+n)$ onto $\binom{[1,m+n]}{m}$. By the composition of this bijection and the inverse of the bijection $\mcY_{m}(m+n) \to \mcF(Y_{m,n})$, $\lambda \mapsto Y_{\lambda}$, we obtain a bijection $I$ from $\mcF(Y_{m,n})$ onto $\binom{[1,m+n]}{m}$. Let $Y \in \mcF(Y_{m,n})$. For $(i,j) \in Y$, we set $H_{Y}(i,j) \coloneqq \{(i,j)\} \cup \{(i,j') \in Y \mid j < j'\} \cup \{(i',j) \in Y \mid i < i'\}$, and call it the {\it hook at $(i,j)$} in $Y$. Also, for $(i,j) \in Y$, we set
\begin{align*}
Y \la i,j \ra \coloneqq \{(i', j') &\mid (i',j') \in Y,\ \text{and}\ i' < i\ \text{or}\ j' < j\} \\
&\cup \{(i'-1,j'-1) \mid (i',j') \in Y,\ i' > i\ \text{and}\ j' > j\}.
\end{align*}
The procedure which obtains $Y\la i,j \ra$ from $Y$ is called {\it removing the hook} at $(i,j)$ from $Y$ (see Figure~\ref{fig} below).

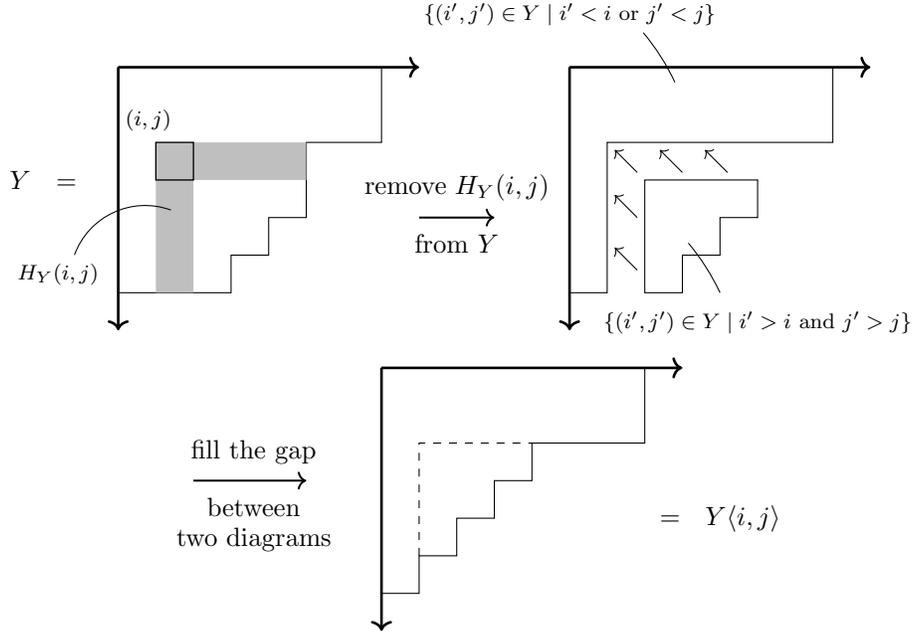
\begin{figure}[htbp]
\centering
\begin{tikzpicture}
\draw(-1,1.5)node{$Y \ \ =$};
\draw(0,0)--(1.5,0)--(1.5,0.5)--(2,0.5)--(2,1)--(2.5,1)--(2.5,2)--(3.5,2)--(3.5,3)--(0,3)--cycle;
\draw[line width = 1pt,->](0,3)--(4,3);
\draw[line width = 1pt,->](0,3)--(0,-0.5);
\draw(0.4,2.3)node{{\footnotesize $(i,j)$}};
\fill[lightgray](0.5,2)--(0.5,0)--(1,0)--(1,1.5)--(2.5,1.5)--(2.5,2)--cycle;
\draw[line width = 0.5pt](0.5,1.5)--(1,1.5)--(1,2)--(0.5,2)--cycle;
\draw(0.7,1.1) arc (70 : 160 : 1);
\draw(-0.8,0.25)node{{\footnotesize $H_{Y}(i,j)$}};

\draw[line width = 0.75pt,->](4,1)--(5,1);
\draw(4.5,1.1)node[above]{remove $H_{Y}(i,j)$};
\draw(4.5,0.9)node[below]{from $Y$};

\draw(6,0)--(6.5,0)--(6.5,2)--(9.5,2)--(9.5,3)--(6,3)--cycle;
\draw(7,0)--(7.5,0)--(7.5,0.5)--(8,0.5)--(8,1)--(8.5,1)--(8.5,1.5)--(7,1.5)--cycle;
\draw[->](6.9,0.3)--(6.6,0.6);
\draw[->](6.9,1)--(6.6,1.3);
\draw[->](7.5,1.6)--(7.2,1.9);
\draw[->](8.1,1.6)--(7.8,1.9);
\draw[->](6.9,1.6)--(6.6,1.9);

\draw[line width = 1pt,->](6,3)--(10,3);
\draw[line width = 1pt,->](6,3)--(6,-0.5);
\draw(6,3.7)node{{\footnotesize $\{(i',j') \in Y \mid i' < i\ \text{or}\ j' < j\}$}};
\draw(7.4,2.8) arc (25 : 40 : 3);
\draw(8.5,-0.4)node{{\footnotesize $\{(i',j') \in Y \mid i' > i\ \text{and}\ j' > j\}$}};
\draw(8,0) arc (25 : 40 : 3);

\draw[line width = 0.75pt,->](1,-2.5)--(2.5,-2.5);
\draw(1.8,-2.4)node[above]{fill the gap};
\draw(1.8,-2.6)node[below]{between};
\draw(1.8,-3)node[below]{two diagrams};
\draw(3.5,-4)--(4,-4)--(4,-3.5)--(4.5,-3.5)--(4.5,-3)--(5,-3)--(5,-2.5)--(5.5,-2.5)--(5.5,-2)--(7,-2)--(7,-1)--(3.5,-1)--cycle;
\draw[dashed](4,-3.5)--(4,-2)--(5.5,-2);
\draw(8,-3)node{$=\ \ Y \la i,j \ra$};
\draw[line width = 1pt,->](3.5,-1)--(7.5,-1);
\draw[line width = 1pt,->](3.5,-1)--(3.5,-4.5);
\end{tikzpicture}
\caption{Removing the hook at $(i,j)$ from $Y$.} \label{fig}
\end{figure}

\subsection{Combinatrial game theory.}\label{subsec22}
For the general theory of combinatorial games, we refer the reader to {\cite[Chapters 1 and 2]{Si}}. In this subsection, we fix an impartial game in normal play whose game positions are all short (in the sense of {\cite[pages 4 and 9]{Si}}).
\begin{defi}\label{outcome}
A game position of an impartial game is called an {\it $\mcN$-position} (resp., a {\it $\mcP$-position}) if the next player (resp., the previous player) has a winning strategy.
\end{defi}
\begin{defi}
For a (proper) subset $X$ of $\BZ_{\geq 0}$, we set $\text{mex}\, X \coloneqq \text{min}\, (\BZ_{\geq 0} \bs X)$.
\end{defi}
For a game position $G$ of an impartial game, we denote by $\mcOH(G)$ the set of all options of  $G$.
\begin{defi}
Let $G$ be a game position. The {\it Grundy value} $\mcG(G)$ of $G$ is defined by
\[
\mcG(G) \coloneqq
\begin{cases}
0 &\text{if}\ G\ \text{is an ending position}, \\
\text{mex}\, \{\mcG(P) \mid P \in \mcOH(G)\} &\text{if}\ G\ \text{is not an ending position}.
\end{cases}
\]
\end{defi}
Recall from {\cite[page 6]{Si}} that each game position of an impartial game is either an $\mcN$-position or a $\mcP$-position.
The following result is well-known in the combinatorial game theory.
\begin{thm}[{\cite[Theorem 2.1]{Si}}]
A game position $G$ is a $\mcP$-position if and only if $\mcG(G) = 0$.
\end{thm}

\section{Unimodal numbering on Young diagrams.}\label{sec3}

Let $Y \in \mcF(Y_{m,n})$. For each box $(i,j) \in Y$, we write $c\,(i,j) \coloneqq \text{min}\, (j-i+m, i-j+n)$ on it; we call this numbering on $Y$ the {\it unimodal numbering} on $Y$.

\begin{exam}\label{EMN}
Assume that $m = 3$ and $n = 5$. The Young diagram $Y = Y_{(4 , 4 , 2)} \in \mcF(Y_{3,5})$ with the unimodal numbering is as follows:
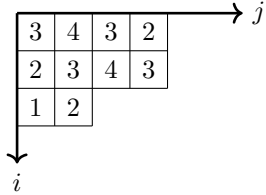
\begin{center}
\begin{tikzpicture}
\draw[line width=1pt,->] (0,0)--(3,0)node[right]{$j$};
\draw[line width=1pt,->] (0,0)--(0,-2)node[below]{$i$};
\draw(0,-0.5)--(2,-0.5);
\draw(0,-1)--(2,-1);
\draw(0,-1.5)--(1,-1.5);
\draw(0.5,0)--(0.5,-1.5);
\draw(1,0)--(1,-1.5);
\draw(1.5,0)--(1.5,-1);
\draw(2,0)--(2,-1);
\draw(0.25,-0.25)node{$3$};
\draw(0.75,-0.25)node{$4$};
\draw(1.25,-0.25)node{$3$};
\draw(1.75,-0.25)node{$2$};
\draw(0.25,-0.75)node{$2$};
\draw(0.75,-0.75)node{$3$};
\draw(1.25,-0.75)node{$4$};
\draw(1.75,-0.75)node{$3$};
\draw(0.25,-1.25)node{$1$};
\draw(0.75,-1.25)node{$2$};
\end{tikzpicture}
\end{center}
\end{exam}

It can be easily checked that $c \coloneqq (m+n-1+\chi)\, /\, 2$ is the maximum number appearing in the unimodal numbering, where
\[
\chi \coloneqq
\begin{cases}
1 &\text{if}\ m+n \in 2\BN, \\
0 &\text{if}\ m+n \in 2\BN+1.
\end{cases}
\]
We define $\mathbb{D}_{m,n} \subset \mathbb{Z}_{\geq 0}^{m+n+1}$ by
\[
\begin{split}
\mathbb{D}_{m,n} \coloneqq \{(a_{1}, a_{2}, a_{3}, \ldots,\ &a_{m+n-1}, a_{m+n}, a_{m+n+1}) \in \mathbb{Z}^{m+n+1}_{\geq 0} \mid \\
a_{1} = a_{m+n+1} &= 0\, ,\, 0 \leq\ a_{k} - a_{k-1} \leq 1\ \text{for}\ 2 \leq k \leq m+1, \\
&0 \leq a_{k} - a_{k+1} \leq 1\ \text{for}\ m+1 \leq k \leq m+n
\}.
\end{split}
\]
For $Y \in \mathcal{F}(Y_{m,n})$, we set $d_{k} = d_{k}(Y) \coloneqq \# \{(i, j) \in Y \mid j - i = -m-1+k\}$ for each $1 \leq k \leq m+n+1$; note that $d_{1} = d_{m+n+1} = 0$. We know from \cite[Proposition~3.6]{AT} that
\[
D_{m,n}(Y) \coloneqq (d_{1}, d_{2}, d_{3}, \ldots , d_{m+n-1}, d_{m+n}, d_{m+n+1})
\]
is an element of $\BD_{m,n}$. Thus we obtain the map $D_{m,n} \colon \mcF(Y_{m,n}) \to \BD_{m,n}$, $Y \mapsto D_{m,n}(Y)$.
An element $D_{m,n}(Y) \in \BD_{m,n}$ is called the {\it diagonal expression} of $Y$. For simplicity of notation, we denote $D_{m,n}$ by $D$.

\begin{exam}\label{EDR}
Assume that $m = 3$ and $n = 5$. Let $\lambda = (4 , 3 , 1) \in \mcY_{3}(8)$. Then we have $D_{3,5}(Y_{\lambda}) = (0, 1, 1, 2, 2, 1, 1, 0, 0) \in \BD_{3,5}$.
\end{exam}

\begin{prop}[{\cite[Proposition~3.6]{AT}}]\label{Dbij}
The map $D_{m,n} \colon \mcF(Y_{m,n}) \to \BD_{m,n}$ is bijective.
\end{prop}

Here we recall from \cite[Subsection~3.3]{AT} the relation between ``removing a hook'' (see Figure~\ref{fig}) and the diagonal expression (see Example~\ref{El,r} below). For a subset $S$ of $Y \in \mcF(Y_{m,n})$, we define $\mcH_{Y}(S)$ to be the multiset consisting of $c\,(i,j)$ for $(i,j) \in S$. The multiset $\mcH_{Y}(S)$ is called the {\it numbering multiset} for $S$. In particular, if $S = H_{Y}(i,j)$ for some $(i,j) \in Y$, then we denote $\mcH_{Y}(S)$ by $\mcH_{Y}(i,j)$. We deduce that $\mcH_{Y}(Y) = \mcH_{Y}(Y \la i,j \ra) \cup \mcH_{Y}(i,j)$ (the union of multisets). Now, let $Y \in \mcF(Y_{m,n})$, and fix $(i,j) \in Y$. Let $i'$ (resp., $j'$) be such that $(i',j) \in Y$ and $(i'+1,j) \notin Y$ (resp., $(i,j') \in Y$ and $(i,j'+1) \notin Y$).
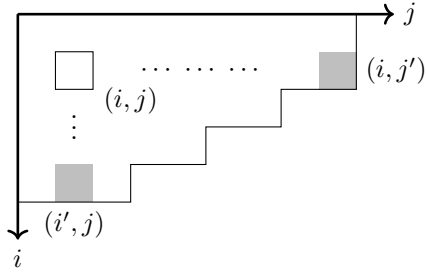
\begin{center}
\begin{tikzpicture}
\draw[line width=1pt,->] (0,0)--(5,0)node[right]{$j$};
\draw[line width=1pt,->] (0,0)--(0,-3)node[below]{$i$};
\fill[lightgray](0.5,-2)--(1,-2)--(1,-2.5)--(0.5,-2.5)--cycle;
\fill[lightgray](4,-0.5)--(4.5,-0.5)--(4.5,-1)--(4,-1)--cycle;
\draw(0,0)--(0,-2.5)--(1.5,-2.5)--(1.5,-2)--(2.5,-2)--(2.5,-1.5)--(3.5,-1.5)--(3.5,-1)--(4.5,-1)--(4.5,0)--cycle;
\draw(0.75,-2.5)node[below]{$(i',j)$};
\draw(4.5,-0.75)node[right]{$(i,j')$};
\draw(0.5,-0.5)--(1,-0.5)--(1,-1)--(0.5,-1)--cycle;
\draw(1.5,-0.85)node[below]{$(i,j)$};
\draw(2.45,-0.75)node{$\cdots$ $\cdots$ $\cdots$};
\draw(0.75,-1.4)node{$\vdots$};
\end{tikzpicture}
\end{center}
Then we see that
\begin{align*}
&\ \ \# \{(x,y) \in Y \mid y-x=-m+k\} - \# \{(x,y) \in Y \la i,j \ra \mid y-x=-m+k\} \\
&=
\begin{cases}
1 &\text{if}\ m+j-i' \leq k \leq m+j'-i, \\
0 &\text{otherwise}.
\end{cases}
\end{align*}
Therefore, if
\begin{align*}
D(Y) = (d_{1}, \ldots , d_{m+j-i'}, \,&d_{m+j-i'+1}, d_{m+j-i'+2}, \ldots , \\
&d_{m+j'-i}, d_{m+j'-i+1}, d_{m+j'-i+2}, \ldots , d_{m+n+1}),
\end{align*}
then
\begin{align*}
D(Y\la i,j \ra) = (d_{1}, \ldots ,\, &d_{m+j-i'}, d_{m+j-i'+1}-1, d_{m+j-i'+2}-1 \ldots , \\
&d_{m+j'-i}-1, d_{m+j'-i+1}-1, d_{m+j'-i+2} \ldots , d_{m+n+1}).
\end{align*}
Thus, if we remove a hook from $Y \in \mcF(Y_{m,n})$, then $1$ is subtracted from some consecutive entries in $D(Y)$; in the case above, the consecutive entries are $d_{l}, d_{l+1}, \ldots, d_{r}$, with $l = m+j-i'+1$ and $r = m+j'-i+1$.
\begin{defi}\label{l,r}
Let $\bold{a} = (a_{1}, a_{2}, \ldots , a_{m+n}, a_{m+n+1}) \in \BD_{m,n}$; recall that $a_{1} = a_{m+n+1} = 0$. For $2 \leq l \leq r \leq m+n$, we write $\bold{a} \xra{l,r} \bold{a}'$ if $a_{k} \geq 1$ for all $l \leq k \leq r$, and $\bold{a}' = (a_{1}, a_{2}, \ldots, a_{l-1}, a_{l}-1, a_{l+1}-1, \ldots , a_{r-1}-1, a_{r}-1, a_{r+1}, \ldots , a_{m+n}, a_{m+n+1}) \in \BZ_{\geq 0}^{m+n+1}$.
\end{defi}
Recall that the map $D = D_{m,n} \colon \mcF(Y_{m,n}) \to \BD_{m,n}$ is bijective. Let $\bold{a}, \bold{a}' \in \BD_{m,n}$, and set $Y \coloneqq D^{-1}(\bold{a}),\, Y' \coloneqq D^{-1}(\bold{a}')$. If $\bold{a} \xra{l,r} \bold{a}'$ for some $2 \leq l \leq r \leq m+n$, then we write $Y \xra{l,r} Y'$.

\begin{exam}\label{El,r}
Keep the notation and setting in Example~\ref{EDR}. It follows that $Y_{\lambda} \la 2,1 \ra = \{(1,1), (1,2), (1,3), (1,4)\}$, and hence $D(Y_{\lambda} \la 2,1 \ra) = (0, 0, 0, 1, 1, 1, 1, 0, 0)$. Thus we have $D(Y_{\lambda}) \xra{2,5} D(Y_{\lambda}\la 2,1 \ra)$ (and hence $Y_{\lambda} \xra{2,5} Y_{\lambda}\la 2,1 \ra$).
\end{exam}

\section{Multiple Hook Removing Game.}\label{sec4}

Abuku and Tada \cite{AT} introduced an impartial game, named Multiple Hook Removing Game (MHRG for short), whose rule is given as follows; recall that $m$ and $n$ are fixed positive integers such that $m \leq n$:
\begin{enumerate}[(1)]
\item All game positions are some Young diagrams contained in $\mcF(Y_{m,n})$ with the unimodal numbering. The starting position is the rectangular Young diagram $Y_{m,n}$.
\item Assume that $Y \in \mcF(Y_{m,n})$ appears as a game position. If $Y \neq \emptyset$ (the empty Young diagram), then a player chooses a box $(i,j) \in Y$, and remove the hook at $(i,j)$ in $Y$; recall from Subsection~\ref{subsec21} that the resulting Young diagram is $Y\la i,j \ra$. Then we know from \cite[Lemma~3.15]{AT} (see also Lemma~\ref{at most} below) that $f \coloneqq \# \{(i',j') \in Y\la i,j \ra \mid \mcH_{Y\la i,j \ra}(i',j') = \mcH_{Y}(i,j)\ \text{(as multisets)}\} \leq 1$. If $f = 0$, then a player moves $Y$ to $Y\la i,j \ra \in \mcOH(Y)$; we call this case and this operation (MHR 1). If $f = 1$, then a player moves $Y$ to $(Y\la i,j \ra)\la i',j' \ra \in \mcOH(Y)$, where $(i', j') \in Y\la i,j \ra$ is the unique element such that $\mcH_{Y\la i,j \ra}(i', j') = \mcH_{Y}(i,j)$; we call this case and this operation (MHR 2).
\item The (unique) ending position is the empty Young diagram $\emptyset$. The winner is the player who makes $\emptyset$ after his/her operation (2).
\end{enumerate}

\begin{defi}\label{state}
We denote by $\mcS(Y_{m,n})$ the set of all those Young diagrams in $\mcF(Y_{m,n})$ which appear as game positions of MHRG (with $Y_{m,n}$ the starting position); in general, $\mcS(Y_{m,n}) \subsetneq \mcF(Y_{m,n})$ as Example~\ref{Emotiv} below shows.
\end{defi}

\begin{defi}\label{MHR}
Let $Y \in \mcS(Y_{m,n})$, and $Y' \in \mcOH(Y)$. If a player moves $Y$ to $Y'$ by operation (MHR 1) (resp., (MHR 2)), then we write $Y \xra{\text{(MHR 1)}} Y'$ (resp., $Y \xra{\text{(MHR 2)}} Y'$).
\end{defi}

\begin{exam}\label{Emotiv}
Assume that $m = 2$ and $n = 3$. The elements of $\mcS(Y_{2,3})$ are
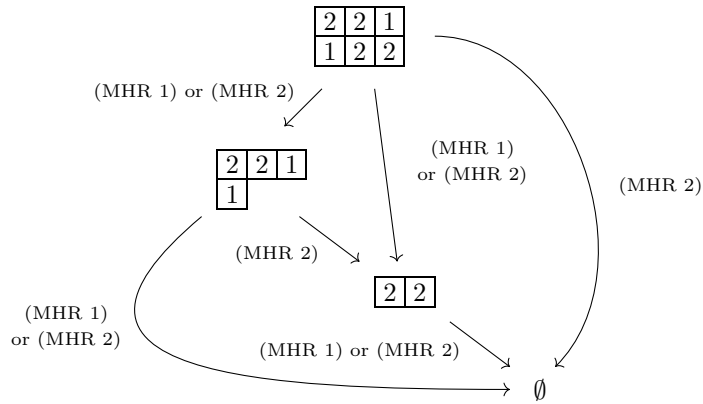
\begin{center}
\begin{tikzpicture}
\draw(0,0)node{$\young(221,122)$};
\draw(-1.3,-1.9)node{$\young(221,1)$};
\draw[->](-0.5,-0.7)--(-1,-1.2);
\draw(0.6,-3.4)node{$\young(22)$};
\draw[->](-0.8,-2.4)--(0,-3);
\draw(-2.2,-1)node[above]{{\scriptsize (MHR 1) or (MHR 2)}};
\draw(-1.1,-2.9)node{{\scriptsize (MHR 2)}};
\draw(2.4,-4.7)node{$\emptyset$};
\draw[->](1.2,-3.8)--(2,-4.4);
\draw(0,-4.2)node{{\scriptsize (MHR 1) or (MHR 2)}};
\draw[->](-2.1,-2.4) to [out=-140, in=-180, looseness = 1.8] (2,-4.7);
\draw(-3.9,-3.7)node{{\scriptsize (MHR 1)}};
\draw(-3.9,-3.8)node[below]{{\scriptsize or (MHR 2)}};
\draw[->](0.2,-0.7)--(0.5,-3);
\draw(1.5,-1.5)node{{\scriptsize (MHR 1)}};
\draw(1.5,-1.6)node[below]{{\scriptsize or (MHR 2)}};
\draw[->](1,0) to [out = 0, in = 45, looseness = 1] (2.6,-4.4);
\draw(4,-2)node{{\scriptsize (MHR 2)}};
\end{tikzpicture}
\end{center}
The following elements of $\mcF(Y_{2,3})$ are not contained in $\mcS(Y_{2,3})$:
\[
\Yvcentermath1 \young(221,12)\ ,\quad \young(22,12)\ ,\quad \young(22,1)\ ,\quad \young(221)\ ,\quad \young(2,1)\ ,\quad \young(2)\ .
\]
\end{exam}

\begin{lemm}[{\cite[Lemma~3.15]{AT}}]\label{at most}
Let $Y \in \mcF(Y_{m,n})$, and $(i,j) \in Y$. Assume that there exists a box $(i',j') \in Y \la i,j \ra$ such that $\mcH_{Y \la i,j \ra}(i',j') = \mcH_{Y}(i,j)$ (as multisets). If $Y \xra{l,r} Y \la i,j \ra$, then $Y \la i,j \ra \xra{\ol{r-1},\ol{l-1}} (Y \la i,j \ra) \la i',j' \ra$. In particular, $\# \{(i',j') \in Y\la i,j \ra \mid \mcH_{Y\la i,j \ra}(i',j') = \mcH_{Y}(i,j)\ (\text{as multisets})\} \leq 1$.
\end{lemm}

\begin{rema}\label{unique}
In fact, the following holds (see \cite[Lemma~3.15]{AT}), although we do not use these facts in this paper.
\begin{enumerate}[(1)]
\item Keep the notation and setting in Lemma~\ref{at most}. There does not exist $(i'', j'') \in (Y\la i,j \ra)\la i',j' \ra$ such that $\mcH_{(Y\la i,j \ra)\la i',j' \ra}(i'', j'') = \mcH_{Y}(i,j)$.
\item Let $(i,j), (k,l) \in Y$. Assume that $\mcH_{Y}(i,j) = \mcH_{Y}(k,l)$. If there exists a box $(i',j') \in Y\la i,j \ra$ such that $\mcH_{Y\la i,j \ra}(i',j') = \mcH_{Y}(i,j)$, then there exists a (unique) box $(k',l') \in Y\la k,l \ra$ such that $\mcH_{Y\la k,l \ra}(k',l') = \mcH_{Y}(i,j)$. Moreover, in this case, we have $(Y\la i,j \ra)\la i',j' \ra = (Y\la k,l \ra)\la k',l' \ra$.
\end{enumerate}
\end{rema}

\section{Description of $\mcS(Y_{m,n})$.}\label{sec5}

Recall that $m,n \in \BN$ are such that $m \leq n$, and that $c = \text{max}\, \{c\,(i,j) \mid (i,j) \in Y_{m,n}\}$ is equal to $(m+n-1+\chi)\, /\, 2$, where $\chi = 0$ (resp., $\chi = 1$) if $m+n$ is odd (resp., even). Also, we have a canonical bijection $I \colon \mcF(Y_{m,n}) \to \binom{[1,m+n]}{m}$ (see Subsection~\ref{subsec21}).

Let $Y \in \mcF(Y_{m,n})$. We set $I_{R}(Y) \coloneqq I(Y) \cap [c+1-\chi, m+n]$; note that $\ol{c+1-\chi} = m+n+1-(c+1-\chi) = c+1 \geq c+1-\chi$.
\begin{thm}\label{main}
Let $Y \in \mcF(Y_{m,n})$, and $\lala$ the partition corresponding to $Y$, that is, $Y = Y_{\lambda}$. The following (I), (I\hspace{-0.2mm}I), (I\hspace{-0.2mm}I\hspace{-0.2mm}I), and (I\hspace{-0.2mm}V) are equivalent. \\
(I) $Y \in \mcS(Y_{m,n})$. \hspace{5mm} (I\hspace{-0.2mm}I) $Y^{D} \in \mcS(Y_{m,n})$. \hspace{5mm} (I\hspace{-0.2mm}I\hspace{-0.2mm}I) $I_{R}(Y) \cap I_{R}(Y^{D}) = \emptyset$. \linebreak[4](I\hspace{-0.2mm}V) $\lambda_{i}+\lambda_{j} \neq n-m+i+j-1$ for all $1 \leq i , j \leq m$.
\end{thm}

The rest of this section is devoted to a proof of Theorem~\ref{main}. We can easily show the following lemma.
\begin{lemm}\label{sub}
\begin{enumerate}[(A)]
\item It holds that $I(Y^{D}) = \{\ol{i} = m+n+1-i \mid i \in I(Y)\} = \ol{I(Y)}$ for $Y \in \mcF(Y_{m,n})$. \label{sub2}
\item Let $Y \in \mcF(Y_{m,n})$, and let $l, r \in [2,m+n]$ such that $l \leq r$. Then, $l-1 \notin I(Y)$ and $r \in I(Y)$ if and only if there exists a (unique) box $(i,j) \in Y$ such that $Y \xra{l,r} Y\la i,j \ra$; in this case, $I(Y\la i,j \ra) = (I(Y) \bs \{r\}) \cup \{l-1\}$ and $I(Y\la i,j \ra^{D}) = (I(Y^{D}) \bs \{\ol{r}\}) \cup \{\ol{l-1}\}$. \label{sub1}\label{sub3}
\end{enumerate}
\end{lemm}

\begin{rema}\label{MHR 2}
Let $Y \in \mcF(Y_{m,n})$, and $(i,j) \in Y$. Let $2 \leq l \leq r \leq m+n$ be such that $Y \xra{l,r} Y\la i,j \ra$. By Lemmas~\ref{at most} and \ref{sub}~(\ref{sub1}), it follows that $\ol{r} \notin I(Y\la i,j \ra)$ and $\ol{l-1} \in I(Y\la i,j \ra)$ if and only if there exists a (unique) box $(i',j') \in Y\la i,j \ra$ such that $Y\la i,j \ra \xra{\ol{r-1}, \ol{l-1}} (Y\la i,j \ra)\la i',j' \ra$; in particular, in this case, it holds that $\mcH_{Y\la i,j \ra}(i',j') = \mcH_{Y}(i,j)$ (as multisets).
\end{rema}

We first show (I) $\Rightarrow$ (I\hspace{-0.2mm}I\hspace{-0.2mm}I). Since $Y \in \mcS(Y_{m,n})$ by (I), there exists a sequence of game positions of the form
\begin{align*}\label{HRS}
Y_{m,n} = Y_{0} \xra{t_{1}} Y_{1} \xra{t_{2}} Y_{2} \xra{t_{3}} \cdots \xra{t_{p}} Y_{p} = Y,
\end{align*}
where $t_{i}$ is either (MHR 1) or (MHR 2) for each $1 \leq i \leq p$. For $1 \leq i \leq p$ such that $t_{i}$ is (MHR 2), we see from Lemmas~\ref{at most} and \ref{sub}~(\ref{sub1}) that $Y_{i-1} \xra{l_{i},r_{i}} Y_{i}' \xra{\ol{r_{i}-1},\ol{l_{i}-1}} Y_{i}$ for some $2 \leq l_{i} \leq r_{i} \leq m+n$ with $l_{i}-1 \notin I(Y_{i-1})$, $r_{i} \in I(Y_{i-1})$, and $Y_{i}' \in \mcF(Y_{m,n})$. Similarly, for $1 \leq i \leq p$ such that $t_{i}$ is (MHR 1), there exists $2 \leq l_{i} \leq r_{i} \leq m+n$ with $l_{i}-1 \notin I(Y_{i-1})$ and $r_{i} \in I(Y_{i-1})$ such that $Y_{i-1} \xra{l_{i},r_{i}} Y_{i}$; we set $Y_{i}' \coloneqq Y_{i}$ by convention.
We show by induction on $p$ that $I_{R}(Y_{p}) \cap I_{R}(Y_{p}^{D}) = \emptyset$. If $p = 0$, then it is obvious that $I_{R}(Y_{m,n}) \cap I_{R}(Y_{m,n}^{D}) = \emptyset$, since $I_{R}(Y_{m,n}) = \{n+1, n+2, \ldots , m+n\}$ and
\[
I_{R}(Y_{m,n}^{D}) = I_{R}(\emptyset) =
\begin{cases}
\emptyset &\text{if}\ m < n, \\
\{m\} &\text{if}\ m = n.
\end{cases}
\]
Assume that $p > 0$; by the induction hypothesis,
\begin{align}\label{IH}
I_{R}(Y_{p-1}) \cap I_{R}(Y_{p-1}^{D}) = \emptyset.
\end{align}
By Lemma~\ref{sub}~(\ref{sub1}), we have
\begin{align}
I_{R}(Y_{p}') \bs \{l_{p}-1\} &= I_{R}(Y_{p-1}) \bs \{r_{p}\}, \label{Y'} \\
I_{R}(Y_{p}'^{D}) \bs \{\ol{l_{p}-1}\} &= I_{R}(Y_{p-1}^{D}) \bs \{\ol{r_{p}}\}. \label{Y'D}
\end{align}

\begin{lemm}\label{neq condition}
It holds that $I_{R}(Y_{p}') \cap I_{R}(Y_{p}'^{D}) \neq \emptyset$ if and only if $\ol{l_{p}-1} \in I(Y_{p-1}) \bs \{r_{p}\}$ or $l_{p}-1 = \ol{l_{p}-1}$; notice that $l_{p}-1 = \ol{l_{p}-1}$ if and only if $\chi = 0$ and $l_{p}-1 = c+1$.
\end{lemm}

\begin{proof}
Assume first that $l_{p}-1 < c+1-\chi$; recall that $\ol{l_{p}-1} > \ol{c+1-\chi} = c+1 \geq c+1-\chi$. It follows from (\ref{Y'}) and (\ref{Y'D}) that
\[
I_{R}(Y_{p}') = I_{R}(Y_{p-1}) \bs \{r_{p}\}\, ,\quad I_{R}(Y_{p}'^{D}) = (I_{R}(Y_{p-1}^{D}) \bs \{\ol{r_{p}}\}) \cup \{\ol{l_{p}-1}\}.
\]
Because $I_{R}(Y_{p-1}) \cap I_{R}(Y_{p-1}^{D}) = \emptyset$ by the induction hypothesis, we see that $I_{R}(Y_{p}') \cap I_{R}(Y_{p}'^{D}) \neq \emptyset$ if and only if $\ol{l_{p}-1} \in I_{R}(Y_{p-1}) \bs \{r_{p}\}$.
Assume next that $l_{p}-1 \geq c+1-\chi$. It follows from (\ref{Y'}) and (\ref{Y'D}) that
\[
I_{R}(Y_{p}') = (I_{R}(Y_{p-1}) \bs \{r_{p}\}) \cup \{l_{p}-1\},
\]
\[
I_{R}(Y_{p}'^{D}) =
\begin{cases}
I_{R}(Y_{p-1}^{D}) \bs \{\ol{r_{p}}\}\ &\text{if}\ \ol{l_{p}-1} < c+1-\chi, \\
(I_{R}(Y_{p-1}^{D}) \bs \{\ol{r_{p}}\}) \cup \{\ol{l_{p}-1}\}\ &\text{if}\ \ol{l_{p}-1} \geq c+1-\chi.
\end{cases}
\]
Here we note that $\ol{l_{p}-1} \in I(Y_{p-1}) \bs \{r_{p}\}$ if and only if $l_{p}-1 \in I(Y_{p-1}^{D}) \bs \{\ol{r_{p}}\}$ by Lemma~\ref{sub} (\ref{sub2}).
If $\ol{l_{p}-1} < c+1-\chi$ (resp., $\ol{l_{p}-1} \geq c+1-\chi$), then it holds that $I_{R}(Y_{p}') \cap I_{R}(Y_{p}'^{D}) \neq \emptyset$ if and only if $\ol{l_{p}-1} \in I(Y_{p-1}) \bs \{r_{p}\}$ (resp., $\ol{l_{p}-1} \in I_{R}(Y_{p-1}) \bs \{r_{p}\}$ or $l_{p}-1 = \ol{l_{p}-1}$). Thus we have proved the lemma.
\end{proof}

\begin{prop}\label{prop12}
\begin{enumerate}[(1)]
\item The operation $t_{p}$ is (MHR 1) if and only if either of the following (a) or (b) holds. \label{tMHR 1}
\begin{enumerate}[(a)]
\item $\ol{l_{p}-1} \notin I(Y_{p-1})$ and $l_{p}-1 \neq \ol{l_{p}-1}$.
\item $l_{p}-1 = \ol{r_{p}}$ (notice that $l_{p}-1 \neq \ol{l_{p}-1}$ also in this case since $l_{p}-1 \neq r_{p} = \ol{l_{p}-1}$).
\end{enumerate}
\item The operation $t_{p}$ is (MHR 2) if and only if $\ol{l_{p}-1} \in I(Y_{p-1}) \bs \{r_{p}\}$ or $l_{p}-1=\ol{l_{p}-1}$. \label{tMHR 2}
\end{enumerate}
\end{prop}

\begin{proof}
It suffices to show only part (2). We first show the ``only if'' part of (2). Assume that $t_{p}$ is (MHR 2); recall that $Y_{p-1} \xra{l_{p},r_{p}} Y_{p}' \xra{\ol{r_{p}-1}, \ol{l_{p}-1}} Y_{p}$. It follows from Lemma~\ref{sub} (\ref{sub1}) (applied to $Y = Y_{p}'$ and $Y\la i,j \ra = Y_{p}$) that $\ol{l_{p}-1} \in I(Y_{p}') = (I(Y_{p-1}) \bs \{r_{p}\}) \cup \{l_{p}-1\}$. Thus we have $\ol{l_{p}-1} \in I(Y_{p-1}) \bs \{r_{p}\}$ or $\ol{l_{p}-1} = l_{p}-1$.
We next show the ``if'' part of (2); by Remark~\ref{MHR 2}, and Lemmas~\ref{at most} and \ref{sub}~(\ref{sub1}), it suffices to show that $\ol{r_{p}} \notin I(Y_{p}')$ and $\ol{l_{p}-1} \in I(Y_{p}')$. Because $I(Y_{p}') = (I(Y_{p-1}) \bs \{r_{p}\}) \cup \{l_{p}-1\}$, it is obvious from the assumption that $\ol{l_{p}-1} \in I(Y_{p}')$. Let us show that $\ol{r_{p}} \notin I(Y_{p}')$. Suppose, for a contradiction, that $\ol{r_{p}} \in I(Y_{p}')$. Since $I(Y_{p}') = (I(Y_{p-1}) \bs \{r_{p}\}) \cup \{l_{p}-1\}$, and since $\ol{r_{p}} \neq l_{p}-1$, we have $\ol{r_{p}} \in I(Y_{p-1}) \bs \{r_{p}\} \subset I(Y_{p-1})$, and hence $r_{p} \in I(Y_{p-1}^{D})$ by Lemma~\ref{sub} (\ref{sub2}). If $c+1-\chi \leq r_{p}$, then $r_{p} \in I_{R}(Y_{p-1}^{D})$. Since $r_{p} \in I_{R}(Y_{p-1})$ by Lemma~\ref{sub} (\ref{sub1}) (applied to $Y_{p-1} \xra{l_{p},r_{p}} Y_{p}'$), we get $r_{p} \in I_{R}(Y_{p-1}) \cap I_{R}(Y_{p-1}^{D})$, which contradicts the induction hypothesis (\ref{IH}). If $c+1-\chi > r_{p}$, then $c+1-\chi \leq c+1 = \ol{c+1-\chi} < \ol{r_{p}}$, which implies that $\ol{r_{p}} \in I_{R}(Y_{p-1})$. Since $r_{p} \in I(Y_{p-1})$, we have $\ol{r_{p}} \in I_{R}(Y_{p-1}^{D})$ by Lemma~\ref{sub} (\ref{sub2}). Hence we get $\ol{r_{p}} \in I_{R}(Y_{p-1}) \cap I_{R}(Y_{p-1}^{D})$, which contradicts the induction hypothesis (\ref{IH}). Therefore we obtain $\ol{r_{p}} \notin I(Y_{p}')$, as desired. Thus we have proved the proposition.
\end{proof}
If $t_{p}$ is (MHR 1) (recall that $Y_{p}' = Y_{p}$ and $Y_{p}^{D\prime} = Y_{p}^{D}$ in this case), then we see by Lemma~\ref{neq condition} and Proposition~\ref{prop12} (\ref{tMHR 1}) that $I_{R}(Y_{p}) \cap I_{R}(Y_{p}^{D}) = \emptyset$.
Assume that $t_{p}$ is (MHR 2), or equivalently, $\ol{l_{p}-1} \in I(Y_{p-1}) \bs \{r_{p}\}$ or $l_{p}-1 = \ol{l_{p}-1}$ by Proposition~\ref{prop12} (\ref{tMHR 2}). Because $Y_{p-1} \xra{l_{p},r_{p}} Y_{p}' \xra{\ol{r_{p}-1}, \ol{l_{p}-1}} Y_{p}$ in this case, it follows from Lemma~\ref{sub} (\ref{sub1}) that
\begin{align}
I_{R}(Y_{p}) \bs \{\ol{r_{p}}, l_{p}-1\} &= I_{R}(Y_{p-1}) \bs \{r_{p}, \ol{l_{p}-1}\}, \label{Y''} \\
I_{R}(Y_{p}^{D}) \bs \{r_{p}, \ol{l_{p}-1}\} &= I_{R}(Y_{p-1}^{D}) \bs \{\ol{r_{p}}, l_{p}-1\}. \label{Y''D}
\end{align}
Hence, by (\ref{Y''}) and (\ref{Y''D}), together with the induction hypothesis (\ref{IH}), we obtain $I_{R}(Y_{p}) \cap I_{R}(Y_{p}^{D}) = \emptyset$. Thus we have proved (I) $\Rightarrow$ (I\hspace{-0.2mm}I\hspace{-0.2mm}I) in Theorem~\ref{main}.

Conversely, we prove (I\hspace{-0.2mm}I\hspace{-0.2mm}I) $\Rightarrow$ (I), that is, $Y \in \mcS(Y_{m,n})$ if $I_{R}(Y) \cap I_{R}(Y^{D}) = \emptyset$. We show by (descending)  induction on $\la I(Y) \ra \coloneqq \sum_{i \in I(Y)} i$. It is obvious that $Y_{m,n} \in \mcS(Y_{m,n})$. Assume that $\la I(Y) \ra < \la I(Y_{m,n}) \ra$. Since $I(Y_{m,n}) = [n+1,m+n]$, and $I(Y) \neq I(Y_{m,n})$ with $\# I(Y) = m$,  there exists $r \notin I(Y)$ such that $n+1 \leq r$. Also, there exists $l \leq r$ such that $l-1 \in I(Y)$; note that $l-1 < r$. Here we show that $\ol{l-1} \notin I(Y)$. Suppose, for a contradiction, that $\ol{l-1} \in I(Y)$. If $c+1-\chi \geq l-1$, then $c+1-\chi \leq c+1 = \ol{c+1-\chi} \leq \ol{l-1}$, and hence $\ol{l-1} \in I_{R}(Y)$. By Lemma~\ref{sub} (\ref{sub2}) applied to $l-1 \in I(Y)$, it follows that $\ol{l-1} \in I_{R}(Y^{D})$. Thus we obtain $\ol{l-1} \in I_{R}(Y) \cap I_{R}(Y^{D})$, which contradicts the assumption that $I_{R}(Y) \cap I_{R}(Y^{D}) = \emptyset$. If $c+1-\chi < l-1$, then $l-1 \in I_{R}(Y^{D})$ because $\ol{l-1} \in I(Y)$. Since $l-1 \in I_{R}(Y)$, we get $l-1 \in I_{R}(Y) \cap I_{R}(Y^{D})$, which contradicts the assumption that $I_{R}(Y) \cap I_{R}(Y^{D}) = \emptyset$. Therefore we obtain $\ol{l-1} \notin I(Y)$.

\begin{prop}\label{hukugen}
Keep the setting above.
\begin{enumerate}[(1)]
\item If $\ol{r} \notin I(Y)$ or $\ol{r} = l-1$, then there exists a (unique) Young diagram $Y'$ such that $I(Y') = (I(Y) \bs \{l-1\}) \cup \{r\}$ and $I(Y'^{D}) = (I(Y^{D}) \bs \{\ol{l-1}\}) \cup \{\ol{r}\}$. Furthermore, $Y' \in \mcS(Y_{m,n})$, and $Y' \xra{\text{(MHR 1)}} Y$.
\item If $\ol{r} \in I(Y)$ and $\ol{r} \neq l-1$, then there exists a (unique) Young diagram $Y''$ such that $I(Y'') = (I(Y) \bs \{\ol{r}, l-1\}) \cup \{r, \ol{l-1}\}$ and $I(Y''^{D}) = (I(Y^{D}) \bs \{r, \ol{l-1}\}) \cup \{\ol{r}, l-1\}$. Furthermore, $Y'' \in \mcS(Y_{m,n})$, and $Y'' \xra{\text{(MHR 2)}} Y$.
\end{enumerate}
\end{prop}

\begin{proof}
(1) Recall that $l-1 \in I(Y)$ and $r \notin I(Y)$, which implies that $(I(Y) \bs \{l-1\}) \cup \{r\} \in \binom{[1,m+n]}{m}$. Since $I \colon \mcF(Y_{m,n}) \to \binom{[1,m+n]}{m}$ is a bijection, there exists unique $Y' \in \mcF(Y_{m,n})$ such that $I(Y') = (I(Y) \bs \{l-1\}) \cup \{r\}$; note that $I(Y'^{D}) = (I(Y^{D}) \bs \{\ol{l-1}\}) \cup \{\ol{r}\}$ by Lemma~\ref{sub} (\ref{sub2}). Then it follows from Lemma~\ref{sub} (\ref{sub1}) that $Y' \xra{l,r} Y$. Because $\ol{r} \notin I(Y)$ or $\ol{r} = l-1$ by the assumption of (1), and $I_{R}(Y) \cap I_{R}(Y^{D}) = \emptyset$ by assumption, it can be easily verified that $I_{R}(Y') \cap I_{R}(Y'^{D}) = \emptyset$. Since $l-1 < r$, we have $\la I(Y') \ra > \la I(Y) \ra$, and hence $Y' \in \mcS(Y_{m,n})$ by the induction hypothesis. Because $\ol{l-1} \notin I(Y)$, we see from Remark~\ref{MHR 2} that there does not exist a box $(i,j) \in Y$ such that $Y \xra{\ol{r-1}, \ol{l-1}} Y\la i,j \ra$. Thus we obtain $Y' \xra{\text{(MHR 1)}} Y$, as desired.

\noindent (2) Let $Y'$ be as in the proof of part (1). Since $\ol{r} \in I(Y)$ and $\ol{r} \neq l-1$ by the assumption of (2), and $\ol{l-1} \notin I(Y)$ as seen above,
\[
(I(Y') \bs \{\ol{r}\}) \cup \{\ol{l-1}\} = (I(Y) \bs \{\ol{r}, l-1\}) \cup \{r, \ol{l-1}\} \in \binom{[1,m+n]}{m}.
\]
Thus there exists $Y'' \in \mcF(Y_{m,n})$ such that $I(Y'') = (I(Y) \bs \{\ol{r}, l-1\}) \cup \{r, \ol{l-1}\}$; note that $I(Y''^{D}) = (I(Y^{D}) \bs \{r, \ol{l-1}\}) \cup \{\ol{r}, l-1\}$ by Lemma~\ref{sub} (\ref{sub2}). It follows from Lemma~\ref{sub} (\ref{sub1}) that $Y'' \xra{\ol{r-1}, \ol{l-1}} Y' \xra{l,r} Y$. Because $\ol{r} \in I(Y)$ and $\ol{r} \neq l-1$ by the assumption of (2), and $I_{R}(Y) \cap I_{R}(Y^{D}) = \emptyset$ by assumption, it can be easily verified that $I_{R}(Y'') \cap I_{R}(Y''^{D}) = \emptyset$. Since $l-1 < r$ and $\ol{l-1} > \ol{r}$, we have $\la I(Y'') \ra > \la I(Y) \ra$, and hence $Y'' \in \mcS(Y_{m,n})$ by the induction hypothesis. We see from Lemma~\ref{at most} that $Y'' \xra{\text{(MHR 2)}} Y$, as desired.
\end{proof}

By Proposition~\ref{hukugen}, we obtain $Y \in \mcS(Y_{m,n})$. This completes the proof of \linebreak[4](I\hspace{-0.2mm}I\hspace{-0.2mm}I) $\Rightarrow$ (I), and hence (I) $\Leftrightarrow$ (I\hspace{-0.2mm}I\hspace{-0.2mm}I). The equivalence (I\hspace{-0.2mm}I) $\Leftrightarrow$ (I\hspace{-0.2mm}I\hspace{-0.2mm}I) follows from the equivalence (I) $\Leftrightarrow$ (I\hspace{-0.2mm}I\hspace{-0.2mm}I) since $I_{R}(Y^{D}) \cap I_{R}((Y^{D})^{D}) = I_{R}(Y) \cap I_{R}(Y^{D})$.

Finally, let us show the equivalence (I\hspace{-0.2mm}I\hspace{-0.2mm}I) $\Leftrightarrow$ (I\hspace{-0.2mm}V). Let $Y \in \mcF(Y_{m,n})$, and $\lala \in \mcY_{m}(m+n)$ be such that $Y = Y_{\lambda}$. We first show (I\hspace{-0.2mm}V) $\Rightarrow$ \linebreak[4](I\hspace{-0.2mm}I\hspace{-0.2mm}I). Obviously, if $I_{R}(Y) \cap I_{R}(Y^{D}) \neq \emptyset$, then $I(Y) \cap I(Y^{D}) \neq \emptyset$. It follows from Subsection~\ref{subsec21} that
\begin{align*}
I(Y) &= \{\lambda_{p}+m-p+1 \mid 1 \leq p \leq m\}, \\
I(Y^{D})&= \{n-\lambda_{q}+q \mid 1 \leq q \leq m\}.
\end{align*}
Hence, $I(Y) \cap I(Y^{D}) \neq \emptyset$ if and only if $\lambda_{i}+m-i+1 = n-\lambda_{j}+j$ (or equivalently, $\lambda_{i}+\lambda_{j} = n-m+i+j-1$) for some $1 \leq i, j \leq m$. Thus we have shown (I\hspace{-0.2mm}V) $\Rightarrow$ (I\hspace{-0.2mm}I\hspace{-0.2mm}I).

We next show (I\hspace{-0.2mm}I\hspace{-0.2mm}I) $\Rightarrow$ (I\hspace{-0.2mm}V). Assume that $\lambda_{i} + \lambda_{j} = n-m+i+j-1$ for some $1 \leq i , j \leq m$; we may assume that $i \leq j$. As seen above, we have $\lambda_{i}+m-i+1 \in I(Y) \cap I(Y^{D})$. Hence it suffices to show that if $\lambda_{i}+\lambda_{j} = n-m+i+j-1$, then $\lambda_{i}+m-i+1 \in [c+1-\chi, m+n]$. Indeed, suppose, for a contradiction, that $\lambda_{i}+m-i+1 \notin [c+1-\chi, m+n]$. Then, $\lambda_{i}+m-i+1 < c+1-\chi$ or $m+n < \lambda_{i}+m-i+1$. Because $\lambda_{i} +m-i+1 \leq n+m-i+1 \leq n+m$, we get $\lambda_{i}+m-i+1 < c+1-\chi$. Since $i \leq j$ (and hence $\lambda_{i} \geq \lambda_{j}$) and $\lambda_{i} < c-m-\chi+i$, we have $\lambda_{i} + \lambda_{j} \leq 2\lambda_{i} < (m+n-1+\chi)-2m-2\chi +2i = n-m-\chi +2i-1 \leq n-m+i+j-1 = \lambda_{i}+\lambda_{j}$, which is a contradiction. Therefore, we conclude that $\lambda_{i}+m-i+1 \in [c+1-\chi, m+n]$. Thus we have shown (I\hspace{-0.2mm}I\hspace{-0.2mm}I) $\Rightarrow$ (I\hspace{-0.2mm}V), thereby completing the proof of (I\hspace{-0.2mm}I\hspace{-0.2mm}I) $\Leftrightarrow$ (I\hspace{-0.2mm}V).

\section{Application.}\label{sec6}
Let $t \in \BZ_{\geq 0}$ and $m,n \in \BN$ such that $t \leq m \leq n$. For $(\lambda_{1}, \ldots , \lambda_{t}) \in \mcY_{t}(t+n)$, we set 
\[
\bracket{\lambda_{1}, \ldots ,\lambda_{t}} \coloneqq (\lambda_{1} , \ldots , \lambda_{t} , \lambda_{t+1} , \ldots , \lambda_{m}) \in \mcY_{m}(m+n),
\]
with $\lambda_{k} \coloneqq 0$ for $t+1 \leq k \leq m$.
\begin{thm}\label{t rows}
Under the notation and setting above, $Y_{\bracket{\lambda_{1}, \ldots , \lambda_{t}}} \in \mcS(Y_{m,n})$ if and only if $Y_{(\lambda_{1} , \ldots , \lambda_{t})} \in \mcS(Y_{t,n-m+t})$. Moreover, the Grundy value of $Y_{\bracket{\lambda_{1}, \ldots ,\lambda_{t}}} \in \mcS(Y_{m,n})$ is equal to the Grundy value of $Y_{(\lambda_{1} , \ldots , \lambda_{t})} \in \mcS(Y_{t,n-m+t})$.
\end{thm}
\begin{proof}
Since $\lambda_{k} = 0$ for $t+1 \leq k \leq m$, it follows from Theorem~\ref{main} that $Y_{\bracket{\lambda_{1}, \ldots ,\lambda_{t}}} \in \mcS(Y_{m,n})$ if and only if $\lambda_{i}+\lambda_{j} \neq n-m+i+j-1$ for all $1 \leq i \leq j \leq t$ and
\begin{equation}\label{sk}
\lambda_{s} \neq n-m+s+k-1\ \text{for all}\ 1 \leq s \leq t\ \text{and}\ t+1 \leq k \leq m;
\end{equation}
note that $0 \neq n-m+k+l-1$ for all $t+1 \leq k,\, l \leq m$ since $m \leq n$. Also, notice that (\ref{sk}) is equivalent to $\lambda_{1} \leq n-m+t$. Therefore, we deduce that $Y_{\bracket{\lambda_{1}, \ldots , \lambda_{t}}} \in \mcS(Y_{m,n})$ if and only if $Y_{(\lambda_{1} , \ldots , \lambda_{t})} \in \mcS(Y_{t,n-m+t})$.

Next, we show the assertion on the Grundy values. Assume that $Y_{(\lambda_{1}, \ldots , \lambda_{t})} \in \mcS(Y_{t,n-m+t})$, or equivalently, $Y_{\bracket{\lambda_{1}, \ldots , \lambda_{t}}} \in \mcS(Y_{m,n})$. If $t = 0$ or $\lambda_{1} = 0$, then $Y_{\bracket{\lambda_{1}, \ldots , \lambda_{t}}} = Y_{(\lambda_{1} , \ldots , \lambda_{t})} = \emptyset$ (the empty Young diagram). Thus, both the Grundy value of $Y_{\bracket{\lambda_{1}, \ldots , \lambda_{t}}} = \emptyset$ in $\mcS(Y_{m,n})$ and the Grundy value of $Y_{(\lambda_{1} , \ldots , \lambda_{t})} = \emptyset$ in $\mcS(Y_{t,n-m+t})$ are equal to $0$.
Assume that $1 \leq t$ and $1 \leq \lambda_{1}$. Since $m \leq n$ and $1 \leq t$, we get $m-t+1 \leq n+t-1$. Hence, we have $c\,(t,1) = \text{min}\,(1-t+m, t-1+n) = m-t+1$. Since $m-t+1 \leq m+\lambda_{1}-1$, and since $\lambda_{1} \leq n-m+t$ as seen above, we have $c\,(1, \lambda_{1}) = \text{min}\,(\lambda_{1}-1+m, 1-\lambda_{1}+n) \geq m-t+1$.
Thus, we obtain $\text{min}\, \{c\,(p,q) \mid (p,q) \in Y_{\bracket{\lambda_{1}, \ldots , \lambda_{t}}}\} \geq m-t+1$:
\begin{figure}[htbp]
\centering
\begin{tikzpicture}
\draw[line width=1pt,->] (0,0)--(4,0)node[right]{$j$};
\draw[line width=1pt,->] (0,0)--(0,-3)node[below]{$i$};
\draw[dashed] (3.6,0)node[above]{$n$}--(3.6,-2.6)--(0,-2.6)node[left]{$m$};
\draw(0,-2)node[left]{$t$}--(2,-2)--(2,-1.5)--(2,-1.5)--(3,-1.5)--(3,0);
\draw(0.32,-1.75)node{\scalebox{0.4}[0.8]{$m-t+1$}};
\draw(0.8,-1.75)node{\scalebox{0.5}[0.8]{$<$}};
\draw(1.1,-1.75)node{\scalebox{0.5}[0.8]{$\cdots$}};
\draw(0.28,-1.4)node{\scalebox{0.5}[0.8]{\rotatebox{90}{$<$}}};
\draw(0.28,-0.9)node{\scalebox{0.5}[0.8]{$\vdots$}};
\draw(0.28,-0.6)node{\scalebox{0.5}[0.8]{\rotatebox{90}{$<$}}};
\draw(0.28,-0.3)node{\scalebox{0.5}[0.8]{$m$}};
\draw(0.5,-0.3)node{\scalebox{0.5}[0.8]{$<$}};
\draw(0.8,-0.3)node{\scalebox{0.5}[0.8]{$\cdots$}};
\draw(1.1,-0.3)node{\scalebox{0.5}[0.8]{$<$}};
\draw(1.3,-0.3)node{\scalebox{0.5}[0.8]{$c$}};
\draw(1.5,-0.3)node{\scalebox{0.5}[0.8]{$>$}};
\draw(1.9,-0.3)node{\scalebox{0.5}[0.8]{$\cdots$}};
\draw(2.2,-0.3)node{\scalebox{0.5}[0.8]{$>$}};
\draw(2.65,-0.3)node{\scalebox{0.5}[0.8]{$c\,(1,\lambda_{1})$}};
\draw(2.6,-0.6)node{\scalebox{0.5}[0.8]{\rotatebox{90}{$>$}}};
\draw(2.6,-0.85)node{\scalebox{0.5}[0.8]{$\vdots$}};
\draw(2,-3.5)node{if $\lambda_{1} > c-m+1$};

\draw[line width=1pt,->] (5,0)--(10,0)node[right]{$j$};
\draw[line width=1pt,->] (5,0)--(5,-3)node[below]{$i$};
\draw[dashed] (9.6,0)node[above]{$n$}--(9.6,-2.6)--(5,-2.6)node[left]{$m$};
\draw(5,-2)node[left]{$t$}--(6.5,-2)--(6.5,-1.5)--(7,-1.5)--(7.5,-1.5)--(7.5,0);
\draw(5.32,-1.75)node{\scalebox{0.4}[0.8]{$m-t+1$}};
\draw(5.8,-1.75)node{\scalebox{0.5}[0.8]{$<$}};
\draw(6.1,-1.75)node{\scalebox{0.5}[0.8]{$\cdots$}};
\draw(5.28,-1.4)node{\scalebox{0.5}[0.8]{\rotatebox{90}{$<$}}};
\draw(5.28,-0.9)node{\scalebox{0.5}[0.8]{$\vdots$}};
\draw(5.28,-0.6)node{\scalebox{0.5}[0.8]{\rotatebox{90}{$<$}}};
\draw(5.28,-0.3)node{\scalebox{0.5}[0.8]{$m$}};
\draw(5.5,-0.3)node{\scalebox{0.5}[0.8]{$<$}};
\draw(6,-0.3)node{\scalebox{0.5}[0.8]{$\cdots\ \cdots$}};
\draw(6.5,-0.3)node{\scalebox{0.5}[0.8]{$<$}};
\draw(7,-0.3)node{\scalebox{0.5}[0.8]{$c\,(1,\lambda_{1})$}};
\draw(7,-0.65)node{\scalebox{0.5}[0.8]{\rotatebox{90}{$<$}}};
\draw(7,-0.9)node{\scalebox{0.5}[0.8]{$\vdots$}};
\draw(7.5,-3.5)node{if $\lambda_{1} \leq c-m+1$};
\end{tikzpicture}
\caption{Numbering of $Y_{\bracket{\lambda_{1}, \ldots , \lambda_{t}}}$ in $\mcS(Y_{m,n})$.} \label{figg}
\end{figure}

\noindent
We notice that
\begin{enumerate}[(i)]
\item in $Y_{\bracket{\lambda_{1}, \ldots , \lambda_{t}}} \in \mcS(Y_{m,n})$ with the unimodal numbering $c\,(p,q)$ for $(p,q) \in Y_{\bracket{\lambda_{1}, \ldots , \lambda_{t}}}$, if we replace $c\,(p,q)$ by $c\,(p,q)-m+t$, then we get $Y_{(\lambda_{1}, \ldots , \lambda_{t})} \in \mcS(Y_{t,n-m+t})$ with the unimodal numbering; \label{replace1}
\item in $Y_{(\lambda_{1}, \ldots , \lambda_{t})} \in \mcS(Y_{t,n-m+t})$ with the unimodal numbering $c'\,(p,q)$ for $(p,q) \in Y_{(\lambda_{1}, \ldots , \lambda_{t})}$, if we replace $c'\,(p,q)$ by $c'\,(p,q)+m-t$, then we get $Y_{\bracket{\lambda_{1}, \ldots , \lambda_{t}}} \in \mcS(Y_{m,n})$ with the unimodal numbering. \label{replace2}
\end{enumerate}
Here we give an example. Let $m = 3, n= 5$, and $t = 2$. Let $\lambda = (3 , 2 , 0) \in \mcY_{3}(8)$. In $Y_{\bracket{3,2}} \in \mcS(Y_{3,5})$ (resp., $Y_{(3,2)} \in \mcS(Y_{2,4})$) with the unimodal numbering $c\,(p,q)$ for $(p,q) \in Y_{\bracket{3,2}}$ (resp., $c'\,(p,q)$ for $(p,q) \in Y_{(3,2)}$), if we replace $c\,(p,q)$ by $c\,(p,q)-1$ (resp., $c'\,(p,q)$ by $c'\,(p,q)+1$), then we get $Y_{(3,2)} \in \mcS(Y_{2,4})$ (resp., $Y_{\bracket{3,2}} \in \mcS(Y_{3,5})$) with the unimodal numbering:
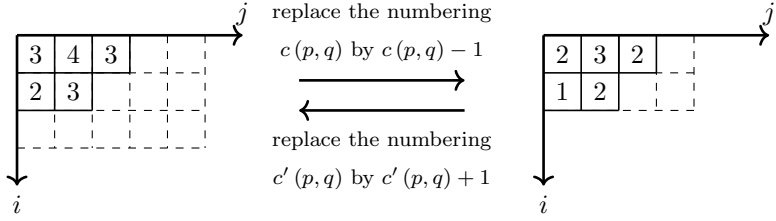
\begin{center}
\begin{tikzpicture}
\draw[line width=1pt,->] (0,0)--(3,0)node[above]{$j$};
\draw[line width=1pt,->] (0,0)--(0,-2)node[below]{$i$};
\draw[dashed](0,-0.5)--(2.5,-0.5);
\draw[dashed](0,-1)--(2.5,-1);
\draw[dashed](0,-1.5)--(2.5,-1.5);
\draw[dashed](2.5,0)--(2.5,-1.5);
\draw[dashed](2,0)--(2,-1.5);
\draw[dashed](1.5,0)--(1.5,-1.5);
\draw[dashed](1,0)--(1,-1.5);
\draw[dashed](0.5,0)--(0.5,-1.5);
\draw[dashed](0,0)--(0,-1.5);
\draw[line width=0.6pt](0.5,0)--(0.5,-1);
\draw[line width=0.6pt](1,0)--(1,-1);
\draw[line width=0.6pt](1.5,0)--(1.5,-0.5);
\draw[line width=0.6pt](0,-0.5)--(1.5,-0.5);
\draw[line width=0.6pt](0,-1)--(1,-1);
\draw(0.25,-0.25)node{$3$};
\draw(0.75,-0.25)node{$4$};
\draw(1.25,-0.25)node{$3$};
\draw(0.25,-0.75)node{$2$};
\draw(0.75,-0.75)node{$3$};
\draw[line width=1pt,->](3.75,-0.6)--(5.95,-0.6);
%\draw(4.75,-0.5)node{{\footnotesize replace the numbering}};
\draw(4.85,0.3)node{{\footnotesize replace the numbering}};
\draw(4.85,-0.2)node{{\footnotesize $c\, (p,q)$ by $c\, (p,q) - 1$}};
\draw[line width=1pt,<-](3.75,-1)--(5.95,-1);
\draw(4.85,-1.4)node{{\footnotesize replace the numbering}};
\draw(4.85,-1.9)node{{\footnotesize $c'\,(p,q)$ by $c'\,(p,q) + 1$}};
\draw[line width=1pt,->] (7,0)--(10,0)node[above]{$j$};
\draw[line width=1pt,->] (7,0)--(7,-2)node[below]{$i$};
\draw[line width=0.6pt](8,0)--(8,-1);
\draw[line width=0.6pt](7.5,0)--(7.5,-1);
\draw[line width=0.6pt](8,0)--(8,-0.5);
\draw[line width=0.6pt](7,-0.5)--(8.5,-0.5);
\draw[line width=0.6pt](7,-1)--(8,-1);
\draw[line width=0.6pt](8.5,0)--(8.5,-0.5);
\draw[dashed](7,-1)--(9,-1);
\draw[dashed](9,0)--(9,-1);
\draw[dashed](9,-0.5)--(8.5,-0.5)--(8.5,-1);
\draw(7.25,-0.25)node{$2$};
\draw(7.75,-0.25)node{$3$};
\draw(8.25,-0.25)node{$2$};
\draw(7.25,-0.75)node{$1$};
\draw(7.75,-0.75)node{$2$};
\end{tikzpicture}
\end{center}
\noindent
It is obvious that the operation (\ref{replace1}) is the inverse of the operation (\ref{replace2}). Moreover, there exists a natural bijection between $\mcOH(Y_{\bracket{\lambda_{1}, \ldots , \lambda_{t}}}) \subset \mcS(Y_{m,n})$ and $\mcOH(Y_{(\lambda_{1}, \ldots , \lambda_{t})}) \subset \mcS(Y_{t,n-m+t})$. Then the inductive argument shows that the Grundy value of $Y_{\bracket{\lambda_{1}, \ldots , \lambda_{t}}}$ in $\mcS(Y_{m,n})$ is equal to the Grundy value of $Y_{(\lambda_{1} , \ldots , \lambda_{t})}$ in $\mcS(Y_{t,n-m+t})$. This completes the proof of Theorem~\ref{t rows}.
\end{proof}

Assume that $m = 2$. Set $c_{i}(q) \coloneqq c+i+4q$ for $i \in \mathbb{Z}$ and $q \geq 0$. We know from \cite[Theorem~4.13]{AT} that a Young diagram $Y_{\lambda} \in \mcS(Y_{2,n})$ with $\lambda = (\lambda_{1}, \lambda_{2})$ is a $\mcP$-position if and only if
\begin{equation}\label{2,n}
\lambda \in
\begin{cases}
\mcC \cup \{(c_{1}(q) , c_{0}(q)), (c_{2}(q) , c_{1}(q)) \mid 0 \leq q \leq (p-1)\, /\, 2\} &\text{if}\ n-2 = 4p, \\
\mcC \cup \{(c_{2}(q) , c_{1}(q)), (c_{3}(q) , c_{2}(q)) \mid 0 \leq q \leq (p-1)\, /\, 2\} &\text{if}\ n-2 = 4p+1, \\
\mcC \cup \{(c_{0}(q) , c_{-1}(q)), (c_{1}(q) , c_{0}(q)) \mid 0 \leq q \leq p\, /\, 2\} &\text{if}\ n-2 = 4p+2, \\
\mcC \cup \{(2p+4 , 2p+2), (2p+5 , 2p+4)\} \\
\hspace{10mm} \cup\ \{(c_{1}(q) , c_{0}(q)), (c_{2}(q) , c_{1}(q)) \mid 1 \leq q \leq p\, /\, 2\} &\text{if}\ n-2=4p+3,
\end{cases}
\end{equation}
where $p \in \mathbb{Z}_{\geq 0}$, and $\mcC = \mcC(p) \coloneqq \{(2q , 2q) \mid 0 \leq q \leq p\}$.

The following is an immediate consequence of Theorem~\ref{t rows} and (\ref{2,n}).
\begin{coro}
We set $d_{i}(q) \coloneqq c-m+2+i+4q$ for $i \in \mathbb{Z}$ and $q \geq 0$. A Young diagram $Y_{\lambda} \in \mcS(Y_{m,n})$ having at most two rows is a $\mcP$-position if and only if
\[
\lambda \in
\begin{cases}
\mathcal{D} \cup \{\bracket{d_{1}(q),d_{0}(q)}, \bracket{d_{2}(q),d_{1}(q)} \mid 0 \leq q \leq (p-1)\, /\, 2\} &\text{if}\ n-m = 4p, \\
\mathcal{D} \cup \{\bracket{d_{2}(q),d_{1}(q)}, \bracket{d_{3}(q),d_{2}(q)} \mid 0 \leq q \leq (p-1)\, /\, 2\} &\text{if}\ n-m = 4p+1, \\
\mathcal{D} \cup \{\bracket{d_{0}(q),d_{-1}(q)}, \bracket{d_{1}(q),d_{0}(q)} \mid 0 \leq q \leq p\, /\, 2\} &\text{if}\ n-m = 4p+2, \\
\mathcal{D} \cup \{\bracket{2p+4,2p+2}, \bracket{2p+5,2p+4}\} \\
\hspace{10mm} \cup\ \{\bracket{d_{1}(q),d_{0}(q)}, \bracket{d_{2}(q),d_{1}(q)} \mid 1 \leq q \leq p\, /\, 2\} &\text{if}\ n-m=4p+3,
\end{cases}
\]
where $p \in \mathbb{Z}_{\geq 0}$, and $\mathcal{D} = \mathcal{D}(p) \coloneqq \{\bracket{2q,2q} \mid 0 \leq q \leq p\}$.
\end{coro}

\noindent
(Yuki Motegi) Graduate School of Pure and Applied Sciences, University of Tsukuba, 1-1-1 Tennodai, Tsukuba, Ibaraki 305-8571, Japan

\noindent
E-mail address: y-motegi@math.tsukuba.ac.jp

\end{document}